\documentclass[11pt,reqno]{amsart}
\usepackage{indentfirst, amssymb,amsmath,amsthm, mathrsfs,cite,graphicx,float,caption,subcaption}
\newtheorem*{theoA}{Theorem A}
\newtheorem*{theoB}{Theorem B}
\newtheorem*{theoC}{Theorem C}

\newtheorem{theo}{Theorem}[section]
\newtheorem{lem}{Lemma}[section]

\newtheorem{rem}{Remark}[section]

\newtheorem{ques}{Question}[section]

\newtheorem{defi}{Definition}[section]

\newtheorem{open problem}{Open problem}[section]
\newcommand{\pa}{\partial}

\newcommand{\la}{\langle}
\newcommand{\ra}{\rangle}
\newcommand{\be}{\begin{equation}}
\newcommand{\ee}{\end{equation}}
\newcommand{\bs}{\begin{small}}
\newcommand{\es}{\end{small}}
\newcommand{\beas}{\begin{eqnarray*}}
\newcommand{\eeas}{\end{eqnarray*}}
\newcommand{\bea}{\begin{eqnarray}}
\newcommand{\eea}{\end{eqnarray}}
\renewcommand{\epsilon}{\varepsilon}
\numberwithin{equation}{section}

\begin{document}
\title[multidimensional refined Bohr's inequalities]{Operator valued analogues of multidimensional refined Bohr's inequalities }
\author[V. Allu, R. Biswas and R. Mandal ]{ Vasudevarao Allu, Raju Biswas and Rajib Mandal}
\date{}
\address{Indian Institute of Technology Bhubaneswar, School of Basic Science, Bhubaneswar-752050, Odisha, India.}
\email{avrao@iitbbs.ac.in}
\address{Department of Mathematics, Raiganj University, Raiganj, West Bengal-733134, India.}
\email{rajubiswasjanu02@gmail.com}
\address{Department of Mathematics, Raiganj University, Raiganj, West Bengal-733134, India.}
\email{rajibmathresearch@gmail.com}
\maketitle
\let\thefootnote\relax
\footnotetext{Type set by \AmS -\LaTeX}
\begin{abstract} 
Let $\mathcal{B}(\mathcal{H})$ denote the Banach algebra of all bounded linear operators acting on complex Hilbert spaces $\mathcal{H}$. 
In this paper, we first establish several sharply refined versions of Bohr's inequality analogues with operator valued functions in the class $\mathcal{B}(\mathbb{D}, \mathcal{B}(\mathcal{H}))$ of bounded analytic functions from the unit disk $\mathbb{D}$ to $\mathcal{B}(\mathcal{H})$ with $\sup_{|z|<1}\Vert f(z)\leq 1$ by utilizing a certain power of the function's norm. 
Additionally, we establish several multidimensional analogues of refined Bohr's inequalities by using operator valued functions in the complete circular domain $\Omega\subset\mathbb{C}^n$. All of the results are sharp.
\end{abstract}
\section{Introduction and Preliminaries} 
The classical Bohr theorem \cite{17} was examined a century ago and has since generated a substantial volume of research activity, which is collectively referred to as the Bohr phenomenon. Let $\mathcal{B}(\mathcal{H})$ denote the Banach algebra of all bounded linear operators acting on complex Hilbert spaces $\mathcal{H}$ with the norm
\beas \Vert A\Vert =\sup_{h\in\mathcal{H}\setminus\{0\}} \frac{\Vert Ah\Vert}{\Vert h\Vert}=\sup_{h\in\mathcal{H},\Vert h\Vert=1}\Vert Ah\Vert, \;\text{where}\;A\in\mathcal{B}(\mathcal{H}).\eeas
Let $H^\infty(\mathbb{D}, X)$ denote the space of bounded analytic functions from the open unit disk $\mathbb{D}:=\{z\in\mathbb{C}:|z|<1\}$ into a complex Banach space $X$ with 
$\Vert f\Vert_{H^\infty(\mathbb{D}, X)}:=\sup_{|z|<1}\Vert f(z)\Vert$. Let $\mathcal{B}(\mathbb{D},X)$ be the class of functions $f$ in $H^\infty(\mathbb{D}, X)$ with $\Vert f\Vert_{H^\infty(\mathbb{D}, X)}\leq 1$. The Bohr radius $R(X)$ \cite{15} for the class $\mathcal{B}(\mathbb{D},X)$ is defined by 
\beas R(X):=\sup\left\{r\in(0,1): \sum_{n=0}^{\infty} \Vert a_n \Vert r^n\leq 1\;\text{for all}\; f(z)=\sum_{n=0}^{\infty} a_n z^n\in\mathcal{B}(\mathbb{D},X), z\in\mathbb{D}\right\}.\eeas 
H. Bohr \cite{17} showed in a remarkable result (improved form) that if $X=\Bbb{C}$, then $R(X)=1/3$, where the norm of $X$ is the usual modulus of complex numbers.
Dixon \cite{21} brought back interest in Bohr's theorem after utilizing it to answer a long-standing query concerning the characterization of Banach algebras satisfying the non-unital 
von Neumann inequality. A substantial body of research has been conducted over the past two decades on a range of topics within complex analysis, including the extensions of 
analytic functions of several complex variables, planar harmonic mappings, polynomials, solutions of elliptic partial differential equations, and other abstract settings. In 1997, Boas 
and Khavinson \cite{16} established the upper and lower bounds of the $n$-dimensional Bohr radius $K_n$ for the Hardy space of bound analytic functions on the unit polydisk.
Further, the authors show that the Bohr radius $K_n$ approaches zero as the domain's dimension $n$ increases to $\infty$. 
In 2006, Defant and Frerick \cite{19} derived an improved lower estimate for $K_n$. Another estimation of $K_n$ was obtained by Defant {\it et al.} {\cite{20} by employing the 
hypercontractivity of the polynomial Bohnenblust-Hille inequality. In 2014, Bayart {\it et al.} \cite{12} derived the exact asymptotic behavior of $K_n$.
In 2019, Popescu \cite{29} expanded the significance of the Bohr inequality for free holomorphic functions to polyballs. 
We refer to \cite{2,3,4,5,6,7,11,13,14,15,22,24,25,28,29,RRS} for additional information on the many aspects and generalizations of the multidimensional Bohr's inequality.
\section{Bohr theorem for operator valued analytic function} 
The primary aim of this paper is to investigate Bohr's inequalities for operator valued analytic functions defined on the unit disk $\mathbb{D}$, specifically for functions in $\mathcal{B}(\mathbb{D}, \mathcal{B}(\mathcal{H}))$. Before we continue, there are some basic notations that need to be fixed. For $T\in \mathcal{B}(\mathcal{H})$, $\Vert T\Vert$ denotes the operator norm of $T$ and the adjoint 
operator $T^*: \mathcal{H} \to\mathcal{H}$ of $T$ is defined by $\la T x, y\ra= \la x, T^*y\ra$ for all $x, y\in\mathcal{H}$. The operator $T$ is said to be normal if $T^*T 
=TT^*$, self-adjoint if $T^*=T$ and positive if $\la T x, x\ra\geq 0$ for all $x\in\mathcal{H}$. The absolute value of $T$ is defined by $|T| := (T^*T)^{1/2}$, while $S^{1/2}$ 
denotes the unique positive square root of a positive operator $S$. Let $I$ be the identity operator on $\mathcal{H}$. Let $f\in H^\infty(\mathbb{D}, \mathcal{B}(\mathcal{H}))$ 
be a bounded analytic function with the expansion
\bea\label{ee1} f(z)=\sum\limits_{n=0}^\infty A_nz^n\;\;\text{for}\;z\in\mathbb{D}, \text{where}\; A_n\in\mathcal{B}(\mathcal{H})\;\text{for all}\; n\in\mathbb{N}\cup\{0\}.\eea
\indent 
For operator valued bounded analytic functions in the unit disk $\Bbb{D}$, Popescu \cite{29} proved the following notable result in 2019. This result is analogous to the classical Bohr theorem.
\begin{theoA}\cite{29}
Let $f\in\mathcal{H}^\infty(\mathbb{D}, \mathcal{B}(\mathcal{H}))$ be an operator valued bounded analytic function with the expansion (\ref{ee1}) such that $A_0= a_0 I$, $a_0\in\mathbb{C}$. Then 
\bea\label{ee2a}\sum\limits_{n=0}^\infty \Vert A_n\Vert r^n\leq \Vert f(z)\Vert_{\mathcal{H}^\infty(\mathbb{D}, \mathcal{B}(\mathcal{H}))}\;\text{for}\;|z|=r\leq 1/3\eea
and $1/3$ is the best possible constant. Moreover, the inequality is strict unless $f$ is a constant. \end{theoA}
If every function in $\mathcal{B}(\mathbb{D}, \mathcal{B}(\mathcal{H}))$ satisfies the inequality (\ref{ee2a}) for $r\leq1/3$, then
$\mathcal{B}(\mathbb{D}, \mathcal{B}(\mathcal{H}))$ satisfies the Bohr phenomenon. It is important to note that the constant $1/3$ is independent of the coefficient of the function.\\[2mm]
\indent 
In recent years, a prominent area of research in the field of one and several complex variables has been the examination of several variations of Bohr's inequality in the context of 
bounded analytic functions. Several improved and refined versions of Bohr’s inequality have been established by different authors (see \cite{1,9,23,25,26,27}). We recall some of them here.\\[2mm]
 Let $f\in\mathcal{H}^\infty(\mathbb{D}, \mathbb{C})$ be of the form 
\bea\label{ee2b}f(z)=\sum\limits_{k=0}^\infty a_k z^k\;\;\text{for}\;z\in\mathbb{D}.\eea
In 2020, Ponnusamy {\it et al.} \cite{30} obtained the following refined Bohr inequality.
\begin{theoB}\cite{30}
Let $f\in\mathcal{B}(\mathbb{D}, \mathbb{C})$ of the form (\ref{ee2b}). Then
\beas \sum_{k=0}^\infty |a_k|r^k+\left(\frac{1}{1+|a_0|}+\frac{r}{1-r}\right)\sum_{k=1}^\infty |a_k|^2 r^{2k}\leq 1\;\;\text{for}\;\;r\leq 1/(2+|a_0|).\eeas 
The numbers $1/(2+|a_0|)$ and $1/(1+|a_0|)$ cannot be improved.
\end{theoB}
Note that $1/3\leq 1/(2+|a_0|)\leq 1/2$. It is evident that the sharp Bohr radius cannot be obtained by directly substituting $a_0$. The improved versions of \textrm{Theorem B} obtained by Ponnusamy {\it et al.} \cite{30} are presented below.
\begin{theoC}\cite{30}
Let $f\in\mathcal{B}(\mathbb{D}, \mathbb{C})$ of the form (\ref{ee2b}). Then
\begin{enumerate}
\item[(a)] $\sum_{k=1}^\infty |a_k|r^k+\left(\frac{1}{1+|a_1|}+\frac{r}{1-r}\right)\sum_{k=2}^\infty |a_k|^2 r^{2k-1}\leq 1$ for $r\leq 3/5$. The number $3/5$ is sharp.
\item[(b)] $\sum_{k=1}^\infty |a_k|r^k+\left(\frac{r^{-1}}{1+|a_1|}+\frac{1}{1-r}\right)\sum_{k=0}^\infty |a_k|^2 r^{2k}\leq 1$ for $r\leq (5-\sqrt{17})/2$. The number $(5-\sqrt{17})/2$ is sharp.
\end{enumerate}
\end{theoC}
\section{Operator valued analogues of multidimensional Bohr inequality}
Unlike the study of the Bohr inequality for certain classes of analytic functions of one complex variable, several researchers investigate this topic with respect to
functions of n-complex variables. Before we proceed further, it is necessary to introduce some key concepts. Let $\alpha$ be an $n$-tuple
$(\alpha_1,\alpha_2,\ldots,\alpha_n)$ of non-negative integers, $|\alpha|$ is the sum $\alpha_1+\alpha_2+\cdots+\alpha_n$ of its components, $\alpha!$ is the product 
$\alpha_1!\alpha_2!\cdots\alpha_n!$, $z$ denotes an $n$-tuple $(z_1,z_2,\cdots,z_n)$ of complex numbers and $z^\alpha$ denotes the product $z_1^{\alpha_1}z_2^{\alpha_2}\cdots z_n^{\alpha_n}$.
Using the standard multi-index notation, we write an operator valued $n$-variable power series
\bea\label{ee2}f(z)=\sum_{\alpha} A_\alpha z^\alpha,\;\;A_\alpha\in\mathcal{B}(\mathcal{H}).\eea
Let $\mathbb{D}^n=\{z\in\mathbb{C}^n: z=(z_1,z_2,\ldots,z_n), |z_j|<1, j=1,2,\ldots,n\}$ be the open unit polydisk. Let $\mathcal{K}_n(\mathcal{H})$ be the largest 
non-negative number such that the $n$-variables power series (\ref{ee2}) converges in $\mathbb{D}^n$ and $\Vert f\Vert_{\mathcal{H}^\infty(\mathbb{D}, \mathcal{B}(\mathcal{H}))}\leq1$. Then  
\bea\label{ee3} \sum_\alpha \Vert A_\alpha\Vert  \left|z^\alpha\right|\leq 1\;\text{for all}\;z\in\mathcal{K}_n(\mathcal{H})\cdot\mathbb{D}^n.\eea 
\begin{defi}\cite{2} A domain $\mathcal{R}\subset\mathbb{C}^n$ is said to be Reinhardt domain centered at $0\in \mathcal{R}$ if for any 
$z=(z_1,z_2,\ldots,z_n)\in\mathcal{R}$, we have that $(z_1e^{i\theta_1},z_2e^{i\theta_2},\ldots, z_ne^{i\theta_n})\in \mathcal{R}$ for each $\theta_k\in[0,2\pi]$ 
$(1\leq k\leq n)$. We say that $\mathcal{R}\subset \mathbb{C}^n$ is a complete Reinhardt domain if for each $z=(z_1, z_2,\ldots, z_n)\in \mathcal{R}$ and for each 
$|u_k|\leq 1$ $(1\leq k\leq n)$, we have that $u.z=(u_1 z_1,u_2 z_2,\ldots,u_n z_n)\in \mathcal{R}$.\end{defi}
\begin{defi}\cite{2} A domain $\Omega\subset\mathbb{C}^n$ is said to be a circular domain centered at $0\in \Omega$ if for any $z=(z_1,z_2,\ldots,z_n)\in \Omega$ and for each $\theta\in[0,2\pi]
$, we have that $e^{i\theta}z=(z_1e^{i\theta},z_2e^{i\theta},\ldots, z_ne^{i\theta})\in \Omega$. We say that $\Omega\subset \mathbb{C}^n$ is a complete 
circular domain at $0\in \Omega$ if for any $z=(z_1,z_2,\ldots,z_n)\in \Omega$ and for each $|u|\leq 1$, we have that $u.z=(u_1z_1,u_2z_2,\ldots,u_nz_n)\in \Omega$.\end{defi}
If $\Omega$ is a complete circular domain centered at $0\in \Omega\subset \mathbb{C}^n$, then every holomorphic function $f$ in $\Omega$ can be expanded into homogeneous polynomials given by
\bea\label{e1} f(z)=\sum_{k=0}^\infty P_k(z) \;\text{for}\;z\in \Omega,\eea
where $P_k(z)$ is a homogeneous polynomial of degree $k$ and $P_0(z)=f(0)$.\\[2mm]
\indent The following questions are expected in such a context.
\begin{ques}\label{Q1} 
Is it possible to establish the analogous results of \textrm{Theorem B} for operator valued functions in $\mathcal{B}(\mathbb{D}, \mathcal{B}(\mathcal{H}))$?
If so, what is the multidimensional analogue of \textrm{Theorem B} for functions in the class $\mathcal{B}(\mathbb{D}, \mathcal{B}(\mathcal{H}))$?\end{ques}
\begin{ques}\label{Q2} 
Is it possible to establish a refined version of \textrm{Theorem B} by employing the operator norm of $f$, {\it i.e.,} $\Vert f\Vert$ for $f\in\mathcal{B}\left(\mathbb{D},\mathcal{B}(\mathcal{H})\right)$?\end{ques}
The primary purpose of this paper is to provide the positive answers to \textrm{Questions \ref{Q1}} and \textrm{\ref{Q2}}.
\section{key lemmas }
The following are key lemmas of this paper and are used to prove the main results. We introduce some notations for $j=1,2$ as follows:
\bea \label{e2}\mathcal{G}_{f}(r):=\sum_{n=1}^\infty \Vert A_n\Vert r^n+\left(\frac{1}{1+\Vert A_1\Vert }+\frac{r}{1-r}\right)\sum_{n=2}^\infty \Vert A_n\Vert ^2r^{2n-1},\\[2mm]
\label{e3} \mathcal{G}_{j,f}(r):=\Vert f(z)\Vert ^j+\sum_{n=1}^\infty \Vert A_n\Vert r^n+\left(\frac{1}{1+\Vert A_1\Vert }+\frac{r}{1-r}\right)\sum_{n=2}^\infty \Vert A_n\Vert ^2r^{2n-1},\\[2mm]
\label{e4}\mathcal{H}_{f}(r):=\sum_{n=1}^\infty \Vert A_n\Vert r^n+\left(\frac{r^{-1}}{1+\Vert A_1\Vert }+\frac{1}{1-r}\right)\sum_{n=1}^\infty \Vert A_n\Vert ^2r^{2n}\;\;\text{and}\\[2mm]
\label{e5}\mathcal{H}_{j,f}(r):=\Vert f(z)\Vert ^j+\sum_{n=1}^\infty \Vert A_n\Vert r^n+\left(\frac{r^{-1}}{1+\Vert A_1\Vert }+\frac{1}{1-r}\right)\sum_{n=1}^\infty \Vert A_n\Vert ^2r^{2n}\eea
for the function $f(z)=\sum\limits_{n=1}^\infty A_nz^n\in\mathcal{B}\left(\mathbb{D},\mathcal{B}(\mathcal{H})\right)$.
\begin{lem}\label{lem1}\cite[counterpart of Schwarz-Pick inequality]{10} Let $B(z)$ be an analytic function with values in $\mathcal{B}(\mathcal{H})$ and satisfying $\Vert B(z)\Vert \leq 1$ on $\mathbb{D}$. Then
\beas (1-|a|)^{n-1}\left\Vert \frac{B^{(n)(a)}}{n!}\right\Vert\leq \frac{\Vert I-B(a)^*B(a)\Vert^{1/2}\Vert I-B(a)B(a)^*\Vert^{1/2}}{1-|a|^2}\eeas 
for each $a\in\mathbb{D}$ and $n\in\mathbb{N}$.\end{lem}
Note that for $f\in\mathcal{B}\left(\mathbb{D},\mathcal{B}(\mathcal{H})\right)$ of the form $f(z)=\sum_{n=0}^\infty A_nz^n$ with $A_0=a_0 I$, $|a_0|<1$. Without loss of generality, assume that $\Vert f\Vert_{H^\infty(\mathbb{D}, \mathcal{B}(\mathcal{H}))}\leq 1$. Then, by virtue of the \textrm{Lemma \ref{lem1}} with $a=0$, we get
\bea\label{s1} \Vert A_n\Vert \leq \Vert I-|A_0|^2\Vert=1-|a_0|^2\quad\text{for}\quad n\in\mathbb{N}.\eea
\begin{lem}\label{lem2}\cite{8} Let $f : \mathbb{D}\to\mathcal{B}(\mathcal{H})$ be an operator valued bounded holomorphic function with the expansion $f(z)=\sum\limits_{n=0}^\infty A_nz^n$ such that $A_k\in\mathcal{B}(\mathcal{H})$ for all $k\in\mathbb{N}\cup\{0\}$ and $A_0=a_0 I$, $|a_0|<1$. If $\Vert f(z)\Vert \leq 1$ in $\mathbb{D}$, then for $z\in\mathbb{D}$, we have \beas \Vert f(z)\Vert \leq \frac{\Vert f(0)\Vert +|z|}{1+\Vert f(0)\Vert |z|}.\eeas\end{lem}
\begin{lem}\label{lem21}\cite{18A}\cite[Lemma B]{30} Suppose that $f\in\mathcal{B}(\mathbb{D},\mathbb{C})$ and $f(z) = \sum_{k=0}^\infty a_kz^k$. Then the following inequalities hold.
\item[(a)] $\vert a_{2k+1}\vert\leq 1-\vert a_0\vert^2-\cdots-\vert a_k\vert^2$, $k=0,1,\ldots$
\item[(b)] $\vert a_{2k}\vert\leq 1-\vert a_0\vert^2-\cdots-\vert a_{k-1}\vert^2-\frac{\vert a_k\vert^2}{1+\vert a_0\vert}$, $k=1,2,\ldots$.
\end{lem}
To prove \textrm{Lemma \ref{lem22}}, we need the following result. 
\begin{rem} We know that, if $T$ is a bounded linear operator on a Hilbert space $X$, then the following conditions are equivalent:
\item[(I)] $T^* T=I$;
\item[(II)] $\langle T x, T y\rangle=\langle x, y\rangle \quad\text{for all}\quad x \quad\text{and}\quad y$;
\item[(III)] $\Vert Tx\Vert=\Vert x\Vert\quad\text{for all}\quad x$.
\end{rem}
In order to prove our main results, we need to establish the analogous result of \textrm{Lemma \ref{lem21}} for functions $f$ belonging to the class $\mathcal{B}\left(\mathbb{D},\mathcal{B}(\mathcal{H})\right)$.
\begin{lem}\label{lem22} Suppose that $f\in\mathcal{B}\left(\mathbb{D},\mathcal{B}(\mathcal{H})\right)$ with the expansion $f(z)=\sum\limits_{n=0}^\infty A_nz^n$ in $\mathbb{D}$ with $A_0=a_0 I$, $|a_0|<1$ such 
that $A_n\in\mathcal{B}(\mathcal{H})$ for all $n\in\mathbb{N}\cup\{0\}$. Then the following inequalities hold.
\item[(a)] $\Vert A_{2k+1}\Vert\leq 1-\Vert A_0\Vert^2-\cdots-\Vert A_k\Vert^2$,\quad $k=0,1,\ldots$
\item[(b)] $\Vert A_{2k}\Vert\leq 1-\Vert A_0\Vert^2-\cdots-\Vert A_{k-1}\Vert^2-\frac{\Vert A_k\Vert^2}{1+\Vert A_0\Vert}$, \quad$k=1,2,\ldots$\end{lem}
\begin{proof} Since $f\in\mathcal{B}\left(\mathbb{D},\mathcal{B}(\mathcal{H})\right)$ of the form $f(z)=\sum_{n=0}^\infty A_nz^n$ with $A_0=a_0 I$, $|a_0|<1$. From (\ref{s1}), we have $\Vert A_n\Vert\leq1-|a_0|^2<1$ for $z\in\mathbb{D}$ and $n\in\mathbb{N}$.
Without loss of generality, let $A_n=a_n I$, where $a_n\in\mathbb{C}$ with $|a_n|<1$ for all $n\in\mathbb{N}\cup\{0\}$. Let us define a function $g(z)=\sum_{k=0}^\infty a_kz^k$. Thus, we have
\beas \Vert f(z)\Vert=\left\Vert I\left(\sum_{n=0}^\infty a_nz^n\right)\right\Vert=\left|\sum_{n=0}^\infty a_nz^n\right|=|g(z)|.\eeas
Since $\Vert f(z)\Vert \leq 1$, we have $|g(z)|\leq 1$ and hence $g\in\mathcal{B}(\mathbb{D},\mathbb{C})$. In view of \textrm{Lemma \ref{lem21}}, 
we get the desired inequalities. 
\end{proof}
The following result provides an operator valued analogue of \textrm{Theorem C (a)}.
\begin{lem}\label{lem3} Suppose that $f\in\mathcal{B}\left(\mathbb{D},\mathcal{B}(\mathcal{H})\right)$ with the expansion $f(z)=\sum\limits_{n=1}^\infty A_nz^n$ in $\mathbb{D}$ such that 
$A_n\in\mathcal{B}(\mathcal{H})$ for all $n\in\mathbb{N}$. Then $\mathcal{G}_{f}(r)\leq 1$ for $|z|=r\leq 3/5$, where $\mathcal{G}_{f}(r)$ is given in (\ref{e2}). The constant 
$3/5$ is the best possible.\end{lem}
\begin{proof}
The function $f$ can be written as $f(z)=zg(z)$, where $g\in\mathcal{B}\left(\mathbb{D},\mathcal{B}(\mathcal{H})\right)$ with the expansion 
$g(z)=\sum\limits_{n=0}^\infty B_nz^n$ such that $B_n\equiv A_{n+1}\in\mathcal{B}(\mathcal{H})$ for all $n\in\mathbb{N}\cup\{0\}$ and $\Vert g(z)\Vert \leq 1$ for $z\in\mathbb{D}$.
Let $g(0)=b_0 I$, where $b_0\in\mathbb{C}$. Since $\Vert g(z)\Vert \leq 1$ for $z\in\mathbb{D}$, so $|b_0|<1$.
In view of (\ref{s1}), we have
\bea\label{r1}&& \Vert B_n\Vert \leq \left\Vert I-\vert B_0\vert ^2\right\Vert \quad\text{for }\quad n\in\mathbb{N}.\eea
Let $\Vert A_1\Vert=\Vert B_0\Vert=|b_0|=a\in[0,1)$, then we have 
\bea\label{r3}\sum_{n=1}^\infty \Vert A_n\Vert r^n=\sum_{n=0}^\infty \Vert B_n\Vert r^{n+1}=r\left(\Vert B_0\Vert+\sum_{n=1}^\infty \Vert B_{2n}\Vert r^{2n}+\sum_{n=0}^\infty \Vert B_{2n+1}\Vert r^{2n+1}\right).\eea
In view of \textrm{Lemma \ref{lem22}} and the inequality (\ref{r1}), we have 
\bea\label{s2}&&\sum_{n=1}^\infty \Vert B_{2n}\Vert r^{2n}+\sum_{n=0}^\infty \Vert B_{2n+1}\Vert r^{2n+1}\nonumber\\[2mm]
&&\leq \sum_{n=1}^\infty\left(1-\sum_{j=0}^{n-1}\Vert B_j\Vert^2-\frac{\Vert B_n\Vert^2}{1+\Vert B_0\Vert}\right)r^{2n}+\sum_{n=0}^\infty\left(1-\sum_{j=0}^{n-1}\Vert B_j\Vert^2-\Vert B_n\Vert^2\right)r^{2n+1}\nonumber\\[2mm]
&&=\sum_{n=1}^\infty r^n-\Vert B_0\Vert ^2\sum_{n=1}^\infty r^n-\frac{1}{1+\Vert B_0\Vert}\sum_{n=1}^\infty \Vert B_n\Vert^2 r^{2n}\nonumber\\[2mm]
&&-\left(\sum_{n=1}^\infty \left(\sum_{j=1}^{n-1}\Vert B_j\Vert^2\right) r^{2n}+\sum_{n=0}^\infty \left(\sum_{j=1}^n\Vert B_j\Vert^2\right) r^{2n+1}\right).\eea
It is easy to see that
\bea\label{s3}&&\sum_{n=1}^\infty \left(\sum_{j=1}^{n-1}\Vert B_j\Vert^2\right) r^{2n}+\sum_{n=0}^\infty \left(\sum_{j=1}^n\Vert B_j\Vert^2\right) r^{2n+1}\nonumber\\[2mm]
&&=\Vert B_1\Vert^2\left(\sum_{k=3}^\infty r^k\right)+\Vert B_2\Vert^2\left(\sum_{k=5}^\infty r^k\right)+\Vert B_3\Vert^2\left(\sum_{k=7}^\infty r^k\right)+\cdots\nonumber\\[2mm]
&&=\frac{\Vert B_1\Vert^2 r^3}{1-r}+\frac{\Vert B_2\Vert^2 r^5}{1-r}+\frac{\Vert B_3\Vert^2 r^7}{1-r}+\cdots=\frac{r}{1-r}\sum_{j=1}^\infty \Vert B_j\Vert^2 r^{2j}.\eea
From (\ref{r3}), (\ref{s2}) and (\ref{s3}), we have
\bea\label{s4}\sum_{n=1}^\infty \Vert A_n\Vert r^n&\leq &ra+(1-a^2)\frac{r^2}{1-r}-\left(\frac{1}{1+a}+\frac{r}{1-r}\right)\sum_{n=1}^\infty \Vert B_n\Vert^2 r^{2n+1}\nonumber\\[2mm]
&=&ra+(1-a^2)\frac{r^2}{1-r}-\left(\frac{1}{1+a}+\frac{r}{1-r}\right)\sum_{n=2}^\infty \Vert A_n\Vert^2 r^{2n-1}.\eea
As $a\in[0,1)$, from (\ref{s4}), we have 
\bea\label{Ra4} \mathcal{G}_{f}(r)&=&\sum_{n=1}^\infty \Vert A_n\Vert r^n+\left(\frac{1}{1+\Vert A_1\Vert }+\frac{r}{1-r}\right)\sum_{n=2}^\infty \Vert A_n\Vert ^2r^{2n-1}\nonumber\\
&\leq&ar+\frac{r^2(1-a^2)}{1-r}=G_1(a,r),\eea
where $G_1(a,r)=ar+(1-a^2)G_2(r)$ with $G_2(r)=r^2/(1-r)\geq 0$.
We claim that $G_1(a,r)\leq 1$ for $r\leq3/5$ and $a\in[0,1)$. 
Differentiating partially $G_1(a,r)$ twice with respect to $a$, we get
\beas\frac{\pa }{\pa a}G_1(a,r)=r-2aG_2(r)\quad\text{and}\quad\frac{\pa^2 }{\pa a^2}G_1(a,r)=-2G_2(r)\leq 0.\eeas
It is clear that $G_1(a,r)$ has a unique critical point at $r/(2G_2(r))$, {\it i.e.}, $(1-r)/(2r)\in[0,1)$ when $r>1/3$ and hence, we have
\beas G_1(a,r)\leq G_1\left(\frac{1-r}{2r},r\right)=\frac{5r^2-2r+1}{4(1-r)}\leq 1\quad\text{for}\quad r\leq 3/5.\eeas
This proves the desired inequality.\\[2mm]
\indent In order to show the sharpness of the constant $3/5$, we consider the function 
\bea\label{s5} \Psi_a(z)=\frac{z(a-z)}{(1-az)}I=\sum_{n=1}^\infty A_nz^n,\eea 
where $A_1=aI$ and $A_n=-(1-a^2)a^{n-2}I$ and for some $a\in[0,1)$. 
By using the function $\Psi_a(z)$, a simple calculation shows that
\beas \mathcal{G}_{\Psi_a}(r)&=& \sum_{n=1}^\infty \Vert A_n\Vert r^n+\left(\frac{1}{1+\Vert A_1\Vert}+\frac{r}{1-r}\right)\sum_{n=2}^\infty \Vert A_n\Vert^2r^{2n-1}\\[2mm]
&=&ar+(1-a^2) \sum_{n=2}^\infty a^{n-2}r^n+(1-a^2)^2\left(\frac{1}{1+a}+\frac{r}{1-r}\right)\sum_{n=2}^\infty a^{2(n-2)}r^{2n-1}\\[2mm]
&=&ar+(1-a^2) \frac{r^2}{1-ar}+\frac{(1-a^2)(1 -a)r^3}{(1-r)(1-ar)}\\[2mm]
&=&\frac{ar(1-r)+(1-a^2)r^2}{(1-r)}=\frac{G_3(a,r)}{(1-r)},\eeas
where $G_3(a,r)=ar(1-r)+(1-a^2)r^2$. 
Comparing this expression with the right-hand side of the expression in formula (\ref{Ra4}) gives the claimed 
sharpness. In order to verify the aforementioned assertion, it is necessary to demonstrate that $G_3(a,r)/(1-r)>1$ for any $r>3/5$ and an appropriate number $a\in(0, 1)$. Now, we have 
\beas \frac{\pa }{\pa a}G_3(a,r)=r(1-r)-2ar^2\geq 0\quad\text{for}\quad a\leq \frac{1-r}{2r}.\eeas
It is easy to see that $(1-r)/(2r)\in(0,1)$ for $r>1/3$ and 
\beas G_3\left(\frac{1-r}{2r},r\right)=\frac{(1-r)^2}{2}+\left(1-\frac{(1-r)^2}{4r^2}\right)r^2=\frac{(1-r)^2+4r^2}{4}>1-r\eeas
for $r>3/5$, which shows that the radius $3/5$ is the best possible. This completes the proof.\end{proof}
The following result provides an operator valued analogue of \textrm{Theorem C (a)} by utilizing the certain power of the operator norm value of functions $f\in\mathcal{B}\left(\mathbb{D},\mathcal{B}(\mathcal{H})\right)$.
\begin{lem}\label{lem4} Suppose that $f\in\mathcal{B}\left(\mathbb{D},\mathcal{B}(\mathcal{H})\right)$ with the expansion $f(z)=\sum_{n=1}^\infty A_nz^n$ in $\mathbb{D}$ such that $A_n\in\mathcal{B}(\mathcal{H})$ for all $n\in\mathbb{N}$. If $\Vert f(z)\Vert \leq 1$ in $\mathbb{D}$, then $\mathcal{G}_{1,f}(r)\leq 1$ for $|z|=r\leq r_1$, where $\mathcal{G}_{1,f}(r)$ is given in (\ref{e3}) and $r_1(\approx 0.484063)$ is the unique positive root of the equation $44 r^4- 68 r^3- 121 r^2+22r+23=0$ in $(0,1)$. The radius $r_1$ is the best possible.\end{lem}
\begin{proof} By using similar arguments as those of \textrm{Lemma \ref{lem3}}, we have $f(z)=zg(z)$, where $g\in\mathcal{B}\left(\mathbb{D},\mathcal{B}(\mathcal{H})\right)$ with the expansion $g(z)=\sum\limits_{n=0}^\infty B_nz^n$ such that $B_n\equiv A_{n+1}\in\mathcal{B}(\mathcal{H})$ for all $n\in\mathbb{N}\cup\{0\}$ and $\Vert g(z)\Vert \leq 1$ for $z\in\mathbb{D}$. Let $g(0)=b_0 I$, where $b_0\in\mathbb{C}$. Since $\Vert g(z)\Vert \leq 1$ for $z\in\mathbb{D}$, so $|b_0|<1$. Let $\Vert A_1\Vert=\Vert B_0\Vert=|b_0|=a\in[0,1)$. In view of \textrm{Lemma \ref{lem2}}, we have 
\be \label{r2}\Vert g(z)\Vert \leq \frac{\Vert g(0)\Vert +|z|}{1+\Vert g(0)\Vert |z|}\quad\text{and}\quad\Vert f(z)\Vert=\left\Vert zg(z)\right \Vert\leq \frac{r(a+r)}{(1+ar)}\quad\text{for}\quad\vert z\vert=r<1.\ee
For $a\in[0,1)$ and by using (\ref{s4}), we deduce that
\bea\label{n1} \mathcal{G}_{1,f}(r)&=&\Vert f(z)\Vert+\sum_{n=1}^\infty \Vert A_n\Vert r^n+\left(\frac{1}{1+\Vert A_1\Vert }+\frac{r}{1-r}\right)\sum_{n=2}^\infty \Vert A_n\Vert ^2r^{2n-1}\nonumber\\[2mm]
&\leq&\frac{r(a+r)}{(1+ar)}+ar+\frac{r^2(1-a^2)}{1-r}\nonumber\\[2mm]
&=&\frac{r(a+r)(1-r)+ar(1+ar)(1-r)+(1-a^2)(1+ar)r^2}{(1+ar)(1-r)}\nonumber\\[2mm]
&=&1+\frac{r(a+r)(1-r)+(ar-1)(1+ar)(1-r)+(1-a^2)(1+ar)r^2}{(1+ar)(1-r)}\nonumber\\[2mm]
&=&1+\frac{G_4(a,r)}{(1+ar)(1-r)},\eea
where $G_4(a,r)=r(a+r)(1-r)+(a^2r^2-1)(1-r)+(1-a^2)(1+ar)r^2$.
Differentiating partially twice with respect to $a$, we get
\be\label{s6}\frac{\pa }{\pa a}G_4(a,r)
=-3r^3a^2-2r^3a+(r^3-r^2+r)\quad\text{and}\quad\frac{\pa^2}{\pa a^2}G_4(a,r)=-2r^3-6ar^3\leq 0,\ee
for all $r,a\in[0,1)$.
This shows that $G_4(a,r)$ has critical points at $(-r\pm \sqrt{4 r^2- 3 r+3})/(3 r)$ and it is evident that $(-r-\sqrt{4 r^2- 3 r+3})/(3 r)\not\in(0,1)$. Also, we see that $a_1=(-r+\sqrt{4 r^2- 3 r+3})/(3 r)\in(0,1)$ for $r>(\sqrt{17}-1)/8$. Thus, we have
\beas G_4(a,r)\leq G_4(a_1,r)=\frac{- 38 r^3  + 63 r^2+ 18 r-27 +( 8 r^2-6r+6) \sqrt{4 r^2-3r +3}}{27}\leq 0\eeas
if
\beas G_5(r):=44 r^4- 68 r^3- 121 r^2+22r+23\geq  0,\eeas
which is true for $r\leq r_1$ and this shown in Fig. \ref{fig3}, where $r_1(\approx 0.484063)$ is the unique positive root of the equation $G_5(r)=0$ between $0$ and $1$.
\begin{figure}[H]
\includegraphics[scale=0.6]{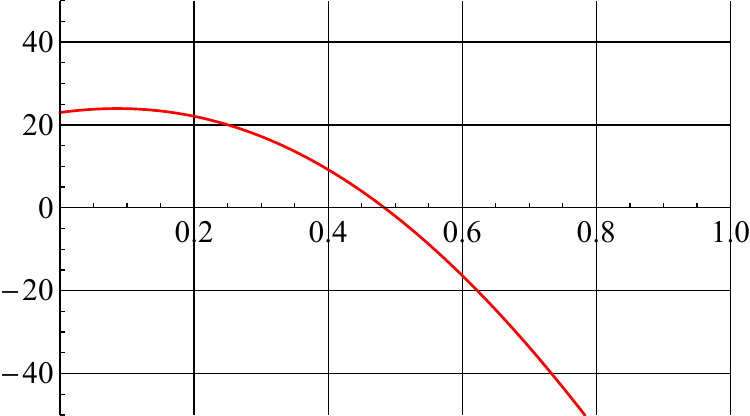}
\caption{The graph of the polynomial $G_5(r)$ for $r\in\left(0,1\right)$}
\label{fig3}
\end{figure}
In order to show that the constant $r_1$ is sharp, we consider the function $\Psi_a(z)$ given in (\ref{s5}).
By using the function $\Psi_a(z)$, a simple calculation shows that
\beas &&\mathcal{G}_{1,\Psi_a}(r)\\
&=& \Vert \Psi_a(-r)\Vert+\sum_{n=1}^\infty \Vert A_n\Vert r^n+\left(\frac{1}{1+\Vert A_1\Vert}+\frac{r}{1-r}\right)\sum_{n=2}^\infty \Vert A_n\Vert^2r^{2n-1}\\[2mm]
&=&\frac{r(a+r)}{(1+ar)}+ar+(1-a^2) \sum_{n=2}^\infty a^{n-2}r^n+(1-a^2)^2\left(\frac{1}{1+a}+\frac{r}{1-r}\right)\sum_{n=2}^\infty a^{2(n-2)}r^{2n-1}\\[2mm]
&=&\frac{r(a+r)}{(1+ar)}+ar+(1-a^2) \frac{r^2}{1-ar}+\frac{(1-a^2)(1 -a)r^3}{(1-r)(1-ar)}\\[2mm]
&=&\frac{r(a+r)}{(1+ar)}+ar+\frac{(1-a^2)r^2}{(1-r)}=1+\frac{G_6(a,r)}{(1+ar)(1-r)},\eeas
where $G_6(a,r)=r(a+r)(1-r)+(a^2r^2-1)(1-r)+(1-a^2)(1+ar)r^2$. Comparing this expression with the right-hand side of the expression in formula (\ref{n1}) gives the claimed 
sharpness. In order to verify the aforementioned assertion, it is necessary to demonstrate that $G_6(a,r)>0$ for any $r>r_1$ and a suitable number $a\in(0, 1)$. From 
(\ref{s6}), we see that $G_6(a,r)$ is an increasing function of $a$ for $a\leq (-r+\sqrt{4 r^2- 3 r+3})/(3 r):=a_1$ and $a_1\in(0,1)$ for $r>(\sqrt{17}-1)/8$. Now, we deduce that
\beas G_6(a_1,r)=\frac{- 38 r^3  + 63 r^2+ 18 r-27 +( 8 r^2-6r+6) \sqrt{ ( 4 r^2-3r +3 )}}{27}\geq 0\eeas
for $r>r_1$ and this shown in Fig. \ref{fig2}, where $r_1(\approx 0.484063)$ is the unique positive root of the equation $G_5(r)=0$ between $0$ and $1$.
\begin{figure}[H]
\includegraphics[scale=0.6]{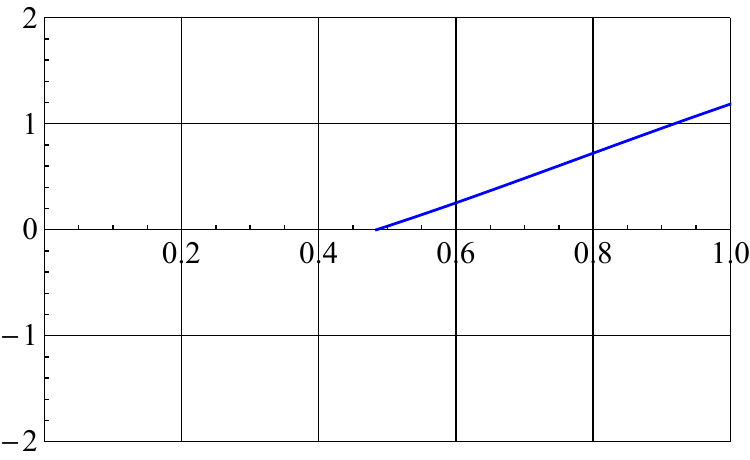}
\caption{The graph of $G_6(a_1,r)$ for $r>r_1$}
\label{fig2}
\end{figure}
This shows that the radius $r_1$ is the best possible. This completes the proof.
\end{proof}
\begin{lem}\label{lem5} Suppose that $f\in\mathcal{B}\left(\mathbb{D},\mathcal{B}(\mathcal{H})\right)$ with the expansion $f(z)=\sum_{n=1}^\infty A_nz^n$ in $\mathbb{D}$ such that $A_n\in\mathcal{B}(\mathcal{H})$ for all $n\in\mathbb{N}$. If $\Vert f(z)\Vert \leq 1$ in $\mathbb{D}$, then $\mathcal{G}_{2,f}(r)\leq 1$ for $|z|=r\leq (\sqrt{5}-1)/2$, where $\mathcal{G}_{2,f}(r)$ is given in (\ref{e3}). The radius $(\sqrt{5}-1)/2$ is the best possible. \end{lem}
\begin{proof} By applying analogous arguments to those presented in \textrm{Lemma \ref{lem4}}, we have
\beas \mathcal{G}_{2,f}(r)&=&\Vert f(z)\Vert^2+\sum_{n=1}^\infty \Vert A_n\Vert r^n+\left(\frac{1}{1+\Vert A_1\Vert }+\frac{r}{1-r}\right)\sum_{n=2}^\infty \Vert A_n\Vert ^2r^{2n-1}\\[2mm]
&\leq&\left(\frac{r(a+r)}{1+ar}\right)^2+ar+\frac{r^2(1-a^2)}{1-r}\\[2mm]
&=&\left(\frac{r(a+r)}{1+ar}\right)^2+\frac{ar(1-r)+(1-a^2)r^2}{(1-r)}=1+G_7(a,r),\eeas
where 
\beas G_7(a,r)=\left(\frac{r(a+r)}{1+ar}\right)^2+\frac{(ar-1)(1-r)+(1-a^2)r^2}{(1-r)}.\eeas 
Differentiating partially $G_7(a,r)$ with respect to $a$, we get
\beas\label{s7} \frac{\pa}{\pa a}G_7(a,r)=\frac{(1 - r)r - 2 a r^2}{1 - r} + \frac{2 r^2 (a + r) (1 - r^2)}{(1 + a r)^3}=\frac{(1 - r)r - 2 a r^2}{1 - r} +\frac{2 r^2 (1 - r)}{(1+r)} G_8(r),
\eeas
where $G_8(r):=(1+r)^2(a + r)/(1 + a r)^3$ for $r\in[0,1)$. Differentiating $G_8(r)$ with respect to $r$, we see that
\beas G'_8(r)&=&-\frac{3 a (1 + r)^2 (a + r)}{(1 + a r)^4} + \frac{(1 + r)^2}{(1 + a r)^3} + \frac{2 (1 + r) (a + r)}{(1 + a r)^3}\\
&=&\frac{(1-a)(1+r)}{(1 + a r)^4}\left(r(a+3)+(1+3a)\right)\geq 0\quad\text{for}\quad r\in[0,1).\eeas
Thus, $G_8(r)$ is a monotonically increasing function of $r$ and it follows that  
\beas G_8(r)\geq G_8(0)=a\quad \text{for all}\quad r\in[0,1)\quad \text{and for}\quad a\in[0,1). \eeas
From (\ref{s7}), we have 
\beas \frac{\pa }{\pa a}G_7(a,r)&\geq& \frac{(1 - r)r - 2 a r^2}{1 - r} +\frac{2 ar^2 (1 - r)}{(1+r)}\\
&=&\frac{(1 - r^2)r - 2 a r^2(1+r)+2 ar^2 (1 - r)^2}{(1-r)(1+r)}=\frac{(1 - r^2)r -2 ar^3 (3-r)}{(1-r)(1+r)}\geq 0.\eeas
if $a\leq (1 - r^2)/(2r^2(3-r))$ and $(1 - r^2)/(2r^2(3-r))\in [0,1)$ for $r>r_2$, where $r_2(\approx 0.401721)$ is the unique positive root of the equation $2 r^3- 7 r^2 +1=0$ in $(0,1)$. 
Consequently, the function $G_7(a,r)$ is a monotonically increasing function of $a$ and this leads to the conclusion that
\beas G_7(a,r)\leq G_7(1,r)=r^2+r-1\leq 0,\quad\text{which is true for}\;r\leq (\sqrt{5}-1)/2.\eeas
\indent To demonstrate the sharpness of the constant $r_1$, we consider the function $\Psi_a(z)$ given in (\ref{s5}).
By utilizing the function $\Psi_a(z)$, a straightforward computation demonstrates that
\beas \mathcal{G}_{2,\Psi_a}(r)&=& \Vert \Psi_a(-r)\Vert^2+\sum_{n=1}^\infty \Vert A_n\Vert r^n+\left(\frac{1}{1+\Vert A_1\Vert}+\frac{r}{1-r}\right)\sum_{n=2}^\infty \Vert A_n\Vert^2r^{2n-1}\\[2mm]
&=&\left(\frac{r(a+r)}{(1+ar)}\right)^2+ar+(1-a^2) \sum_{n=2}^\infty a^{n-2}r^n\\
&&+(1-a^2)^2\left(\frac{1}{1+a}+\frac{r}{1-r}\right)\sum_{n=2}^\infty a^{2(n-2)}r^{2n-1}\\[2mm]
&=&\left(\frac{r(a+r)}{(1+ar)}\right)^2+ar+(1-a^2) \frac{r^2}{1-ar}+\frac{(1-a^2)(1 -a)r^3}{(1-r)(1-ar)}\\[2mm]
&=&\left(\frac{r(a+r)}{(1+ar)}\right)^2+ar+\frac{(1-a^2)r^2}{(1-r)}=1+G_9(a,r),\eeas
where 
\beas G_9(a,r)=\left(\frac{r(a+r)}{1+ar}\right)^2+\frac{(ar-1)(1-r)+(1-a^2)r^2}{(1-r)}\eeas 
and it is easy to see that $\lim_{a\to 1^{-}}G_9(a,r)=r^2+r-1>0$ for $r>(\sqrt{5}-1)/2$. This demonstrates that the radius $(\sqrt{5}-1)/2$ represents the best possible value. This completes the proof. \end{proof}
The following result provides an operator valued analogue of \textrm{Theorem C (b)}.
\begin{lem}\label{lem6} Suppose that $f\in\mathcal{B}\left(\mathbb{D},\mathcal{B}(\mathcal{H})\right)$ with the expansion $f(z)=\sum_{n=1}^\infty A_nz^n$ in $\mathbb{D}$ such that $A_n\in\mathcal{B}(\mathcal{H})$ for all $n\in\mathbb{N}$. If $\Vert f(z)\Vert \leq 1$ in $\mathbb{D}$, then $\mathcal{H}_{f}(r)\leq 1$ for $|z|=r\leq (5-\sqrt{17})/2$, where $\mathcal{H}_{f}(r)$ is given in (\ref{e4}). The radius $(5-\sqrt{17})/2$ is the best possible. \end{lem}
\begin{proof} By applying analogous arguments to those presented in \textrm{Lemma \ref{lem3}} and from (\ref{s4}), we have
\bea\label{n3} &&\sum_{n=1}^\infty \Vert A_n\Vert r^n+\frac{1}{1-r}\sum_{n=1}^\infty \Vert A_n\Vert ^2r^{2n}\nonumber\\[2mm]
&&\leq ra+(1-a^2)\frac{r^2}{1-r}-\left(\frac{1}{1+a}+\frac{r}{1-r}\right)\sum_{n=2}^\infty \Vert A_n\Vert^2 r^{2n-1}+\frac{1}{1-r}\sum_{n=1}^\infty \Vert A_n\Vert ^2r^{2n}\nonumber\\[2mm]
&&= ar+\frac{r^2}{1-r}-\frac{r^{-1}}{1+a}\sum_{n=1}^\infty \Vert A_n\Vert^2 r^{2n}+\frac{r^{-1}}{1+a}\Vert A_1\Vert^2 r^{2}.\eea
From (\ref{n3}), we deduce that
\beas \mathcal{H}_{f}(r)&=&\sum_{n=1}^\infty \Vert A_n\Vert r^n+\left(\frac{r^{-1}}{1+\Vert A_1\Vert }+\frac{1}{1-r}\right)\sum_{n=1}^\infty \Vert A_n\Vert ^2r^{2n}\\[2mm]
&\leq& ar+\frac{r^2}{1-r}+\frac{ra^2}{1+a}=\frac{ar(1-r)(1+a)+r^2(1+a)+a^2r(1-r)}{(1-r)(1+a)}\\[2mm]
&=&1+\frac{\Phi(a,r)}{(1-r)(1+a)},\eeas
where $\Phi(a,r)=2r(1-r)a^2+a(2r-1)+r^2+r-1$. Differentiating partially $\Phi(a,r)$ with respect to $a$, we get
\beas\label{n2} \frac{\pa}{\pa a}\Phi(a,r)=4r(1-r)a+2r-1\quad\text{and}\quad \frac{\pa^2}{\pa a^2}\Phi(a,r)=4r(1-r)\geq 0.\eeas
This implies that $\Phi(a,r)$ is a convex function of $a$ and this leads to the conclusion that
\beas \Phi(a,r)\leq \max\{\Phi(0,r),\Phi(1,r)\}=\left\{\begin{array}{lll}
\Phi(0,r)=r^2+r-1&\text{for}\;r\leq (2-\sqrt{2})/2,\\[2mm]
\Phi(1,r)=-r^2+5r-2&\text{for}\;r>(2-\sqrt{2})/2\end{array}\right.\eeas
and this shown in Fig. \ref{fig1}.
\begin{figure}[H]
\includegraphics[scale=0.7]{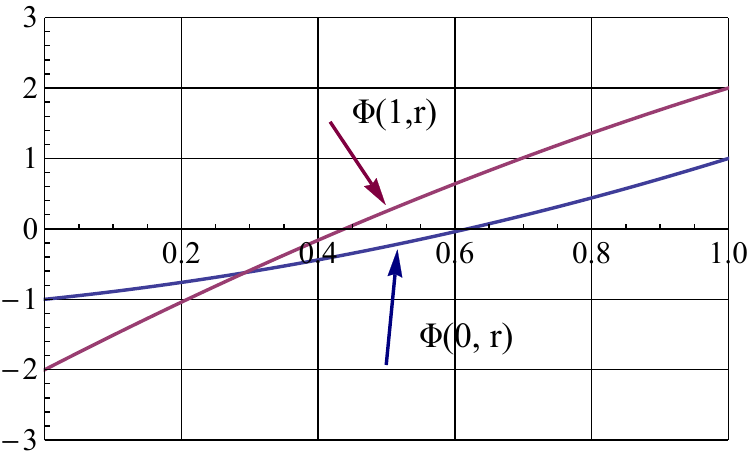}
\caption{The graph of the polynomials $\Phi(0,r)$ and $\Phi(1,r)$ for $0\leq r<1$}
\label{fig1}
\end{figure}
It is easy to see that $\Phi(0,r)\leq 0$ for $r\leq (\sqrt{5}-1)/2$ and $\Phi(1,r)\leq 0$ for $r\leq (5-\sqrt{17})/2$. Thus in any case $\Phi(a,r)\leq 0$ if $r\leq (5-\sqrt{17})/2$.\\[2mm]
\indent In order to show the sharpness of the constant $(5-\sqrt{17})/2$, we consider the function $\Psi_a(z)$ given in (\ref{s5}).
By utilizing the function $\Psi_a(z)$, a straightforward computation demonstrates that
\beas \mathcal{H}_{\Psi_a}(r)&=&\sum_{n=1}^\infty \Vert A_n\Vert r^n+\left(\frac{r^{-1}}{1+\Vert A_1\Vert}+\frac{1}{1-r}\right)\sum_{n=1}^\infty \Vert A_n\Vert^2r^{2n}\\
&=& ar+(1-a^2) \sum_{n=2}^\infty a^{n-2}r^n\\
&&+\left(\frac{r^{-1}}{1+a}+\frac{1}{1-r}\right)\left(a^2r^2+(1-a^2)^2r^4\sum_{n=2}^\infty (a^2r^2)^{n-2}\right)\\
&=& ar+(1-a^2) \frac{r^2}{1-ar}+\left(\frac{r^{-1}}{1+a}+\frac{1}{1-r}\right)\left(a^2r^2+\frac{(1-a^2)^2r^4}{1-a^2r^2}\right)=\Phi_1(a,r).\eeas
It is easy to see that
\beas \lim_{a\to1^{-}}\Phi_1(a,r)=\frac{3r-r^2}{2-2r}>1\quad\text{for}\quad r>(5-\sqrt{17})/2.\eeas
This demonstrates that the radius $(5-\sqrt{17})/2$ represents the best possible value. This completes the proof.\end{proof}
The following results gives an operator valued analogue of \textrm{Theorem C (b)} with a certain power of the operator norm value of functions $f\in\mathcal{B}\left(\mathbb{D},\mathcal{B}(\mathcal{H})\right)$.
\begin{lem}\label{lem7} Suppose that $f\in\mathcal{B}\left(\mathbb{D},\mathcal{B}(\mathcal{H})\right)$ with the expansion $f(z)=\sum_{n=1}^\infty A_nz^n$ in $\mathbb{D}$ such that $A_n\in\mathcal{B}(\mathcal{H})$ for all $n\in\mathbb{N}$. If $\Vert f(z)\Vert \leq 1$ in $\mathbb{D}$, then $\mathcal{H}_{1,f}(r)\leq 1$ for $|z|=r\leq 1/3$, where $\mathcal{H}_{1,f}(r)$ is given in (\ref{e5}). The radius $1/3$ is the best possible. \end{lem}
\begin{proof} 
By applying analogous arguments to those presented in \textrm{Lemmas \ref{lem4}} and from (\ref{n3}), we can deduce that
\beas \mathcal{H}_{1,f}(r)&=&\Vert f(z)\Vert+\sum_{n=1}^\infty \Vert A_n\Vert r^n+\left(\frac{r^{-1}}{1+\Vert A_1\Vert }+\frac{1}{1-r}\right)\sum_{n=1}^\infty \Vert A_n\Vert ^2r^{2n}\\
&\leq& \frac{r(a+r)}{(1+ar)}+ar+\frac{r^2}{1-r}+\frac{ra^2}{1+a}=\Phi_2(a,r).\eeas
Differentiating partially $\Phi_2(a,r)$ with respect to $a$, we get
\bea\label{n4} \frac{\pa}{\pa a}\Phi_2(a,r)&=&r - \frac{a^2 r}{(1 + a)^2}+\frac{2 a r}{1 + a} - \frac{r^2 (a + r)}{(1 + a r)^2} + \frac{r}{1 + a r}\nonumber\\[2mm]
&=&r+ \frac{a r}{(1 + a)^2}(2+a)+ \frac{r}{(1 + a r)^2}\left(1-r^2\right)\geq 0.\eea
This implies that $\Phi_2(a,r)$ is a monotonically increasing function of $a$ and this leads to the conclusion that
\beas \Phi_2(a,r)\leq\Phi_2(1,r)=5r/2+\frac{r^2}{1-r}=\frac{-3 r^2+5 r }{2 (1-r)}\leq 1\quad\text{for}\quad r\leq 1/3.\eeas
In order to show that the constant $1/3$ is sharp, we consider the function $\Psi_a(z)$ given in (\ref{s5}).
By utilizing the function $\Psi_a(z)$, a straightforward computation demonstrates that
\beas \mathcal{H}_{\Psi_a}(r)&=&\Vert \Psi_a(-r)\Vert+\sum_{n=1}^\infty \Vert A_n\Vert r^n+\left(\frac{r^{-1}}{1+\Vert A_1\Vert}+\frac{1}{1-r}\right)\sum_{n=1}^\infty \Vert A_n\Vert^2r^{2n}\\[2mm]
&= &\frac{r(a+r)}{(1+ar)}+ar+(1-a^2) \sum_{n=2}^\infty a^{n-2}r^n\\
&&+\left(\frac{r^{-1}}{1+a}+\frac{1}{1-r}\right)\left(a^2r^2+(1-a^2)^2r^4\sum_{n=2}^\infty (a^2r^2)^{n-2}\right)\\[2mm]
&=& \frac{r(a+r)}{(1+ar)}+ar+(1-a^2) \frac{r^2}{1-ar}\\[2mm]
&&+\left(\frac{r^{-1}}{1+a}+\frac{1}{1-r}\right)\left(a^2r^2+\frac{(1-a^2)^2r^4}{1-a^2r^2}\right)=\Phi_3(a,r).\eeas
It is easy to see that
\beas \lim_{a\to1^{-}}\Phi_3(a,r)=\frac{-3 r^2+5 r }{2 (1-r)}>1\;\text{for}\;r>1/3.\eeas
This demonstrates that the radius $1/3$ represents the best possible value. This completes the proof.\end{proof}
\begin{lem}\label{lem8} Suppose that $f\in\mathcal{B}\left(\mathbb{D},\mathcal{B}(\mathcal{H})\right)$ with the expansion $f(z)=\sum_{n=1}^\infty A_nz^n$ in $\mathbb{D}$ such that $A_n\in\mathcal{B}(\mathcal{H})$ for all $n\in\mathbb{N}$. If $\Vert f(z)\Vert \leq 1$ in $\mathbb{D}$, then $\mathcal{H}_{2,f}(r)\leq 1$ for $|z|=r\leq r_2$, where $\mathcal{H}_{2,f}(r)$ is given in (\ref{e5}) and $r_2(\approx 0.393401)$ is the unique positive root of the equation $2 r^3 - r^2 -5r+2=0$ in $(0,1)$. The radius $r_2$ is the best possible. \end{lem}
\begin{proof} 
By applying analogous arguments to those presented in \textrm{Lemma \ref{lem7}} and from (\ref{n3}), we can deduce that
\beas \mathcal{H}_{1,f}(r)&=&\Vert f(z)\Vert^2+\sum_{n=1}^\infty \Vert A_n\Vert r^n+\left(\frac{r^{-1}}{1+\Vert A_1\Vert }+\frac{1}{1-r}\right)\sum_{n=1}^\infty \Vert A_n\Vert ^2r^{2n}\\[2mm]
&\leq& \left(\frac{r(a+r)}{(1+ar)}\right)^2+ar+\frac{r^2}{1-r}+\frac{ra^2}{1+a}=\Phi_4(a,r).\eeas
Differentiating partially $\Phi_4(a,r)$ with respect to $a$, we get
\bea\label{n4} \frac{\pa}{\pa a}\Phi_4(a,r)&=&r - \frac{a^2 r}{(1+a)^2} + \frac{2 a r}{1 + a} - \frac{2 r^3 (a + r)^2}{(1 + a r)^3} + \frac{2 r^2 (a + r)}{(1 + a r)^2}\nonumber\\[2mm]
&=& r+\frac{a r}{(1 + a)^2}(2+a)+ \frac{2 r^2(a+r)}{(1 + a r)^3}\left(1-r^2\right)\geq 0.\eea
This implies that $\Phi_4(a,r)$ is an increasing function of $a$ and this leads to the conclusion that
\beas \Phi_4(a,r)\leq\Phi_4(1,r)=\frac{-2 r^3 +r^2 +3r}{2 (1-r)}\leq 1\quad\text{for}\quad r\leq r_1,\eeas
where $r_2(\approx 0.393401)$ is the unique positive root of the equation $2 r^3 - r^2 -5r+2=0$ between $0$ and $1$.
In order to show that the constant $r_2$ is sharp, we consider the function $\Psi_a(z)$ given in (\ref{s5}).
By utilizing the function $\Psi_a(z)$, a straightforward computation demonstrates that
\beas\mathcal{H}_{2,\Psi_a}(r)&=&\Vert \Psi_a(-r)\Vert^2+\sum_{n=1}^\infty \Vert A_n\Vert r^n+\left(\frac{r^{-1}}{1+\Vert A_1\Vert}+\frac{1}{1-r}\right)\sum_{n=1}^\infty \Vert A_n\Vert^2r^{2n}\\
&=& \left(\frac{r(a+r)}{(1+ar)}\right)^2+ar+(1-a^2) \frac{r^2}{1-ar}\\&&+\left(\frac{r^{-1}}{1+a}+\frac{1}{1-r}\right)\left(a^2r^2+\frac{(1-a^2)^2r^4}{1-a^2r^2}\right)=\Phi_5(a,r).\eeas
It is easy to see that
\beas \lim_{a\to1^{-}}\Phi_5(a,r)=\frac{-2 r^3 +r^2 +3r}{2 (1-r)}>1\;\text{for}\;r>r_2.\eeas
This demonstrates that the radius $r_2$ represents the best possible value. This completes the proof.\end{proof}
\section{main Results}
Before we present the main results of the paper, we will introduce some notational conventions for $j = 1,2$:
\bea \label{f1}&& \mathcal{T}_{f}(r):=\sum_{k=1}^\infty \Vert P_k(z)\Vert+\left(\frac{r^{-1}}{1+\Vert P_1(z)\Vert}+\frac{r}{1-r}\right)\sum_{k=2}^\infty \Vert P_k(z)\Vert^2,\\
\label{f2}&& \mathcal{T}_{j,f}(r):=\Vert f(z)\Vert^j+\sum_{k=1}^\infty \Vert P_k(z)\Vert+\left(\frac{r^{-1}}{1+\Vert P_1(z)\Vert}+\frac{1}{1-r}\right)\sum_{k=2}^\infty \Vert P_k(z)\Vert^2,\\
\label{f3}&&\mathcal{R}_{f}(r):=\sum_{k=1}^\infty \Vert P_k(z)\Vert+\left(\frac{r^{-1}}{1+\Vert P_1(z)\Vert}+\frac{1}{1-r}\right)\sum_{k=1}^\infty \Vert P_k(z)\Vert^2\;\;\text{and}\\
\label{f4}&&\mathcal{R}_{j,f}(r):=||f(z)||^j+\sum_{k=1}^\infty \Vert P_k(z)\Vert+\left(\frac{r^{-1}}{1+\Vert P_1(z)\Vert}+\frac{1}{1-r}\right)\sum_{k=1}^\infty \Vert P_k(z)\Vert^2.\eea
\indent In the following, we prove the multidimensional refined versions of \textrm{Theorem C (a)} for the operator valued analytic functions in the complete circular domain $\Omega$.
\begin{theo}\label{T1} Assume that the series (\ref{e1}) converges in the domain $\Omega$ such that $\Vert f(z)\Vert<1$ for all $z\in\Omega$ and $f(0)=0$. Then 
\beas \sum_{k=1}^\infty \Vert P_k(z)\Vert+\left(\frac{r^{-1}}{1+\Vert P_1(z)\Vert}+\frac{r}{1-r}\right)\sum_{k=2}^\infty \Vert P_k(z)\Vert^2\eeas
holds in 
the homothetic domain $(1/3)\cdot\Omega$. Moreover, if $\Omega$ is convex, then $1/3$ is the best possible.\end{theo}
\begin{theo}\label{T2} Assume that the series (\ref{e1}) converges in the domain $\Omega$ such that $\Vert f(z)\Vert<1$ for all $z\in\Omega$ and $f(0)=0$. Then 
$\mathcal{T}_{1,f}(r)\leq 1$ holds in the homothetic domain $r_1\cdot\Omega$, where $\mathcal{T}_{1,f}(r)$ is given in (\ref{f2}) and $r_1(\approx 0.484063)$ is the unique positive root of the equation $44 r^4- 68 r^3- 121 r^2+22r+23=0$ in $(0,1)$. Moreover, if $\Omega$ is convex, then $r_1$ is the best possible.\end{theo}
\begin{theo}\label{T3} Assume that the series (\ref{e1}) converges in the domain $\Omega$ such that $\Vert f(z)\Vert<1$ for all $z\in\Omega$ and $f(0)=0$. Then 
$\mathcal{T}_{2,f}(r)\leq 1$ holds 
in the homothetic domain $((\sqrt{5}-1)/2)\cdot\Omega$, where $\mathcal{T}_{2,f}(r)$ is given in (\ref{f2}). Moreover, if $\Omega$ is convex, then $(\sqrt{5}-1)/2$ is the best possible.\end{theo}
In the following, we prove the multidimensional refined versions of \textrm{Theorem C (b)} for the operator valued analytic functions in the complete circular domain $\Omega$.
\begin{theo}\label{T4} Assume that the series (\ref{e1}) converges in the domain $\Omega$ such that $\Vert f(z)\Vert<1$ for all $z\in\Omega$ and $f(0)=0$. Then 
$\mathcal{R}_{f}(r)\leq 1$ holds 
in the homothetic domain $((5-\sqrt{17})/2)\cdot\Omega$, where $\mathcal{R}_{f}(r)$ is given in (\ref{f3}). Moreover, if $\Omega$ is convex, then $(5-\sqrt{17})/2$ is the best possible.\end{theo}
\begin{theo}\label{T5} Assume that the series (\ref{e1}) converges in the domain $Q$ such that $\Vert f(z)\Vert<1$ for all $z\in\Omega$ and $f(0)=0$. Then 
$\mathcal{R}_{1,f}(r)\leq 1$ holds in 
the homothetic domain $(1/3)\cdot\Omega$, where $\mathcal{R}_{1,f}(r)$ is given in (\ref{f4}). Moreover, if $\Omega$ is convex, then $1/3$ is the best possible.\end{theo}
\begin{theo}\label{T6} Assume that the series (\ref{e1}) converges in the domain $Q$ such that $\Vert f(z)\Vert<1$ for all $z\in\Omega$ and $f(0)=0$. Then 
$\mathcal{R}_{2,f}(r)\leq 1$ holds in 
the homothetic domain $r_2\cdot\Omega$, where $\mathcal{R}_{2,f}(r)$ is given in (\ref{f4}) and $r_2(\approx 0.393401)$ is the unique positive root of the equation $2 r^3 - r^2 -5r+2=0$ in $(0,1)$. Moreover, if $\Omega$ is convex, then $r_1$ is the best possible.\end{theo}

\section{Proofs of the main results}
\begin{proof}[\bf{Proof of \textrm{Theorem \ref{T1}}}]
In order to derive inequality (\ref{f1}), we first convert the multidimensional power series (\ref{e1}) into a power series of one complex variable and then apply \textrm{Lemma \ref{lem3}}. In each section of the domain $\Omega$ by a complex line
\bea\label{Ra1} \mathcal{L}:=\{z=(z_1,z_2,\ldots,z_n) : z_j=\alpha_j t, j=1,2,\ldots,n,\;t\in\mathbb{C}\},\eea
the series (\ref{e1}) with $f(0)=0$ turns into the power series in the complex variable $t$:
\bea\label{Ra2} f(\alpha t)=\sum_{k=1}^\infty P_k(\alpha) t^k.\eea
Since $\Vert f(\alpha t)\Vert<1$ for all $t\in\mathbb{D}$, so in view of \textrm{Lemma \ref{lem3}}, we have
\bea\label{Ra3} \sum_{k=1}^\infty \left\Vert P_k(\alpha) t^k\right\Vert+\left(\frac{r^{-1}}{1+\Vert P_1(\alpha )t\Vert}+\frac{1}{1-r}\right)\sum_{k=2}^\infty \left\Vert P_k(\alpha)t^k\right\Vert^2\leq 1\eea
for $z$ in the section $\mathcal{L}\cap\left(\frac{3}{5}\Omega\right)$ and for $r\leq 3/5$. It should be noted that $\mathcal{L}$ is an arbitrary complex line passing through the 
origin. Thus, the inequality (\ref{Ra3}) is nothing but the desired inequality $ \mathcal{T}_{f}(r)\leq 1$.\\[2mm]
\indent In order to demonstrate the sharpness of the constant $3/5$, let us consider the domain $\Omega$ to be convex. Then $\Omega$ is an intersection of half spaces
\beas \Omega=\bigcap_{\alpha\in\mathcal{J}}\left\{z=(z_1,z_2,\ldots,z_n) : \text{Re}\left(\sum_{j=1}^n \alpha_j z_j\right)<1\right\}\;\;\text{for some}\; \mathcal{J}.\eeas
Since $\Omega$ is a circular domain, so we have 
\beas \Omega=\bigcap_{\alpha\in\mathcal{J}}\left\{z=(z_1,z_2,\ldots,z_n) :\left\vert\sum_{j=1}^n \alpha_j z_j\right\vert<1\right\}\;\;\text{for some}\; \mathcal{J}.\eeas 
Consequently, in order to demonstrate that the constant $3/5$ represents the best possible value, it is sufficient to illustrate that no improvement can be made to $3/5$ in each domain $\Omega_\alpha=\left\{z=(z_1,z_2,\ldots,z_n) :\left\vert\sum_{j=1}^n \alpha_j z_j\right\vert<1\right\}$.
Since $3/5$ is the best possible in \textrm{Lemma \ref{lem3}}, there exists an analytic function $\chi_1: \mathbb{D}\to \mathcal{B}(\mathcal{H})$ defined by 
\beas \chi_1(\eta)=\frac{\eta(3/5-\eta)}{1-3\eta/5}I=\sum_{k=1}^\infty A_k \eta^k\;\;\text{for}\;\eta\in\mathbb{D},\eeas
where $A_1=(3/5)I$ and $A_k=-(16/25)(3/5)^{k-2}I$ for $k\geq 2$ such that $\Vert \chi_1(\eta)\Vert<1$ in the unit disk $\mathbb{D}$ with $\chi_1(0)=0$. But for any $\vert \eta\vert =r >3/5$, the inequality $\mathcal{G}_{\chi_1}(r)\leq 1$ fails to be hold in the disk $\mathbb{D}_r=\{z: \vert z\vert<r\}$. We now consider the function 
$\varphi :\Omega_\alpha \to \mathbb{D}$ defined by $\varphi(z) = \alpha_1 z_1+ \alpha_2 z_2+\ldots + \alpha_n z_n$. Thus, the function $f_1(z) = (\chi_1\circ\varphi)(z)$ gives the 
sharpness of the constant $3/5$ for each domain $\Omega_\alpha$. This completes the proof.
\end{proof}
\begin{proof}[\bf{Proof of \textrm{Theorem \ref{T2}}}]
In the light of \textrm{Lemma \ref{lem4}} and the analogous proof of \textrm{Theorem \ref{T1}}, it is easy to obtain the inequality $\mathcal{T}_{1,f}(r)\leq 1$ in the homothetic
domain $r_1\cdot\Omega$, where $r_1(\approx 0.484063)$ is the unique positive root of the equation $44 r^4- 68 r^3- 121 r^2+22r+23=0$ in $(0,1)$. \\[2mm]
\indent In order to demonstrate that the constant $r_1$ represents the best possible value in the case where $\Omega$ is convex, it is sufficient to provide an analogue of the proof presented in \textrm{Theorem \ref{T1}} and to demonstrate that $r_1$ cannot be improved for each domain
$\Omega_\alpha = \{z = (z_1,z_2,\ldots, z_n) : \vert \alpha_1 z_1+ \alpha_2 z_2+\ldots+\alpha_n z_n\vert < 1\}$. Since $r_1$ is the best possible in \textrm{Lemma \ref{lem4}}, 
there exists an analytic function $\chi_2: \mathbb{D} \to \mathcal{B}(\mathcal{H})$ such that $\Vert \chi_2(\eta)\Vert < 1$ in $\mathbb{D}$, but the inequality 
$\mathcal{G}_{1,\chi_2}(r)\leq 1$
fails to hold in the disk $\mathbb{D}_r$ for each $\vert \eta\vert = r > r_1$. Thus, the function $f_2(z) = (\chi_2\circ\varphi)(z)$ gives the sharpness of the constant $r_1$ in each 
domain $\Omega_\alpha$, where $\varphi : \Omega_\alpha\to\mathbb{D}$ is defined by $\varphi(z)=\alpha_1 z_1+ \alpha_2 z_2+\ldots+\alpha_n z_n$. This completes the proof.
\end{proof}
\begin{proof}[\bf{Proof of \textrm{Theorem \ref{T3}}}]
In the light of \textrm{Lemma \ref{lem5}} and the analogues proof of \textrm{Theorem \ref{T1}},  we can easily obtain the inequality $\mathcal{T}_{2,f}(r)\leq 1$ in the 
homothetic domain $((\sqrt{5}-1)/2)\cdot\Omega$.\\[2mm]
\indent To prove the constant $r_1$ is the best possible whenever $\Omega$ is convex, in view of the analogues
proof of \textrm{Theorem \ref{T1}}, it is enough to show that $r_1$ cannot be improved for each domain
$\Omega_\alpha = \{z = (z_1,z_2,\ldots, z_n) : \vert \alpha_1 z_1+ \alpha_2 z_2+\ldots+\alpha_n z_n\vert < 1\}$. Since $(\sqrt{5}-1)/2$ is the best possible in 
\textrm{Lemma \ref{lem5}}, there exists an analytic function $\chi_3: \mathbb{D} \to \mathcal{B}(\mathcal{H})$ such that $\Vert \chi_3(\eta)\Vert < 1$ in $\mathbb{D}$, but the 
inequality $\mathcal{G}_{2,\chi_3}(r)\leq 1$ fails to hold in the disk $\mathbb{D}_r$ for each $\vert \eta\vert = r > r_1$. Thus, the function $f_3(z) = (\chi_3\circ\varphi)(z)$ gives the 
sharpness of the constant $(\sqrt{5}-1)/2$ in each 
domain $\Omega_\alpha$, where $\varphi : \Omega_\alpha\to\mathbb{D}$ is defined by $\varphi(z)=\alpha_1 z_1+ \alpha_2 z_2+\ldots+\alpha_n z_n$. This completes the proof.
\end{proof}
\begin{proof}[\bf{Proof of \textrm{Theorem \ref{T4}}}]
In the light of \textrm{Lemma \ref{lem6}} and the analogues proof of \textrm{Theorem \ref{T1}}, we can easily obtain the inequality $\mathcal{R}_f(r)\leq 1$ in the homothetic
domain $((5-\sqrt{17})/2)\cdot\Omega$.\\[2mm] 
\indent To prove
the constant $(5-\sqrt{17})/2$ is the best possible whenever $\Omega$ is convex, in view of the analogues
proof of \textrm{Theorem \ref{T1}}, it is enough to show that $(5-\sqrt{17})/2$ cannot be improved for each domain
$\Omega_\alpha = \{z = (z_1,z_2,\ldots, z_n) : \vert \alpha_1 z_1+ \alpha_2 z_2+\ldots+\alpha_n z_n\vert < 1\}$. Since $(5-\sqrt{17})/2$ is the best possible in \textrm{Lemma \ref{lem6}}, 
there exists an analytic function $\chi_4: \mathbb{D} \to \mathcal{B}(\mathcal{H})$ such that $\Vert \chi_4(\eta)\Vert < 1$ in $\mathbb{D}$, but the inequality 
$\mathcal{H}_{\chi_4}(r)\leq 1$ fails to hold in the disk $\mathbb{D}_r$ for each $\vert \eta\vert = r > (5-\sqrt{17})/2$. Thus, the function $f_4(z) = (\chi_4\circ\varphi)(z)$ gives the sharpness of the 
constant $(5-\sqrt{17})/2$ in each domain $\Omega_\alpha$, where $\varphi : \Omega_\alpha\to\mathbb{D}$ is defined by $\varphi(z)=\alpha_1 z_1+ \alpha_2 z_2+\ldots+\alpha_n z_n$. This completes the proof.
\end{proof}
\begin{proof}[\bf{Proof of \textrm{Theorem \ref{T5}}}]
In the light of \textrm{Lemma \ref{lem7}} and the analogues proof of \textrm{Theorem \ref{T1}}, we can easily obtain the inequality $\mathcal{R}_{1,f}(r)\leq 1$ in the homothetic
domain $(1/3)\cdot\Omega$.\\[2mm] 
\indent To prove the constant $1/3$ is the best possible whenever $\Omega$ is convex, in view of the analogues
proof of \textrm{Theorem \ref{T1}}, it is enough to show that $1/3$ cannot be improved for each domain
$\Omega_\alpha = \{z = (z_1,z_2,\ldots, z_n) : \vert \alpha_1 z_1+ \alpha_2 z_2+\ldots+\alpha_n z_n\vert < 1\}$. Since $1/3$ is the best possible in \textrm{Lemma \ref{lem7}}, 
there exists an analytic function $\chi_5 : \mathbb{D} \to \mathcal{B}(\mathcal{H})$ such that $\Vert \chi_5(\eta)\Vert < 1$ in $\mathbb{D}$, but $\mathcal{H}_{1,\chi_5}(r)\leq 1$
fails to hold in the disk $\mathbb{D}_r$ for each $\vert \eta\vert = r > 1/3$. Thus, the function $f_5(z) = (\chi_5\circ\varphi)(z)$ gives the sharpness of the constant $1/3$ in each 
domain $\Omega_\alpha$, where $\varphi : \Omega_\alpha\to\mathbb{D}$ is defined by $\varphi(z)=\alpha_1 z_1+ \alpha_2 z_2+\ldots+\alpha_n z_n$. This completes the proof.
\end{proof}
\begin{proof}[\bf{Proof of \textrm{Theorem \ref{T6}}}]
In the light of \textrm{Lemma \ref{lem8}} and the analogues proof of \textrm{Theorem \ref{T1}},  we can easily obtain the inequality $\mathcal{R}_{2,f}(r)\leq 1$ in the homothetic
domain $r_2\cdot\Omega$, where $r_2(\approx 0.393401)$ is the unique positive root of the equation $2 r^3 - r^2 -5r+2=0$ in $(0,1)$.\\[2mm]
\indent To prove
the constant $r_2$ is the best possible whenever $\Omega$ is convex, in view of the analogues
proof of \textrm{Theorem \ref{T1}}, it is enough to show that $r_2$ cannot be improved for each domain
$\Omega_\alpha = \{z = (z_1,z_2,\ldots, z_n) : \vert \alpha_1 z_1+ \alpha_2 z_2+\ldots+\alpha_n z_n\vert < 1\}$. Since $r_1$ is the best possible in \textrm{Lemma \ref{lem4}}, 
there exists an analytic function $\chi_6: \mathbb{D} \to \mathcal{B}(\mathcal{H})$ such that $\Vert \chi_6(\eta)\Vert < 1$ in $\mathbb{D}$, but $\mathcal{H}_{2,\chi_6}(r)\leq 1$
fails to hold in the disk $\mathbb{D}_r$ for each $\vert \eta\vert = r > r_2$. Thus, the function $f_6(z) = (\chi_6\circ\varphi)(z)$ gives the sharpness of the constant $r_2$ in each 
domain $\Omega_\alpha$, where $\varphi : \Omega_\alpha\to\mathbb{D}$ is defined by $\varphi(z)=\alpha_1 z_1+ \alpha_2 z_2+\ldots+\alpha_n z_n$. This completes the proof.
\end{proof}
\section{ Statements \& Declarations}
\noindent{\bf Acknowledgment:} The work of the second Author is supported by University Grants Commission (IN) fellowship (No. F. 44 - 1/2018 (SA - III)). The authors like to thank the anonymous reviewers and and the editing team for their valuable suggestions towards the improvement of the paper.\\[2mm]
{\bf Conflict of Interest:} The authors declare that there are no conflicts of interest regarding the publication of this paper..\\
{\bf Data availability:} Not applicable.


\begin{thebibliography}{33}
\bibitem{1}{\sc M. B. Ahamed, V. Allu and H. Halder}, The Bohr Phenomenon for analytic functions on shifted disks, {\it Ann. Fenn. Math.} {\bf47} (2022), 103-120.
\bibitem{2} {\sc L. Aizenberg}, Multidimensional analogues of Bohr’s theorem on power series, {\it Proc. Amer. Math. Soc.} {\bf128} (2000), 1147-1155. 
\bibitem{3} {\sc L. Aizenberg, A. Aytuna and P. Djakov}, Generalization of theorem on Bohr for bases in spaces of holomorphic functions of several complex variables, {\it J. Math. Anal.Appl.} {\bf258} (2001), 429-447.
\bibitem{4} {\sc L. Aizenberg}, Generalization of results about the Bohr radius for power series, {\it Stud. Math.} {\bf180} (2007), 161-168.
\bibitem{5} {\sc S. A. Alkhaleefah, I.R. Kayumov and S. Ponnusamy}, On the Bohr inequality with a fixed zero coefficient, {\it Proc. Amer. Math. Soc.} {\bf147} (2019), 5263-5274.
\bibitem{6} {\sc V. Allu and H. Halder}, Bohr phenomenon for certain subclasses of Harmonic Mappings, {\it Bull. Sci. Math.} {\bf173} (2021), 103053.
\bibitem{7}  {\sc V. Allu and H. Halder}, Bohr radius for certain classes of starlike and convex univalent functions, {\it J. Math. Anal. Appl.} {\bf493}(1) (2021), 124519.
\bibitem{8}{\sc V. Allu and H. Halder}, Operator valued analogue of multidimensional Bohr’s inequality, {\it Canadian Math. Bull.} {\bf65}(4) (2022), 1020-1035.
\bibitem{9} {\sc V. Allu and H. Halder}, Bohr operator on operator valued polyanalytic functions on simply connected domains, {\it Canadian Math. Bull.} {\bf66}(4) (2023), 1411-1422.
\bibitem{10} {\sc J. M. Anderson and J. Rovnyak}, On Generalized Schwarz-Pick Estimates, {\it Mathematika} {\bf53} (2006), 161-168.
\bibitem{11} {\sc A. Aytuna and P. Djakov}, Bohr property of bases in the space of entire functions and its generalizations, {\it Bull. London Math. Soc.} {\bf45}(2)(2013), 411-420.
\bibitem{12} {\sc F. Bayart, D. Pellegrino and J. B. Seoane-Sep\'ulveda}, The Bohr radius of the $n$-dimensional polydisk is equivalent to $\sqrt{(\log n)/n}$, {\it Adv. Math.} {\bf264} (2014), 726-746.
\bibitem{13} {\sc C. B\'en\'eteau, A. Dahlner and D. Khavinson}, Remarks on the Bohr phenomenon, {\it Comput. Methods Funct. Theory} {\bf4}(1) (2004), 1-19.
\bibitem{14} {\sc B. Bhowmik and N. Das}, Bohr phenomenon for operator-valued functions, {\it Proc. Edinburgh Math. Soc.} {\bf64}(1) (2021), 72-86.
\bibitem{15} {\sc O. Blasco}, The Bohr radius of a Banach space, In Vector measures, integration and related topics, 5964, {\it Oper. Theory Adv. Appl.}, 201, Birkh\"auser Verlag, Basel, 2010.
\bibitem{16} {\sc H. P. Boas and D. Khavinson}, Bohr’s power series theorem in several variables, {\it Proc. Amer. Math. Soc.} {\bf125} (1997), 2975-2979.
\bibitem{17} {\sc H. Bohr}, A theorem concerning power series, {\it Proc. Lond. Math. Soc.} {\bf s2-13} (1914), 1-5.
\bibitem{18A} {\sc F. Carlson}, Sur les coefficients d'une fonction born\'ee dans le cercle unit\'e (French), {\it Ark. Mat. Astr. Fys.} {\bf 27A}(1) (1940), 8 pp.
\bibitem{19}  {\sc A. Defant and L. Frerick}, A logarithmic lower bound for multi-dimensional Bohr radii, {\it Israel J. Math.} {\bf152} (2006), 17-28.
\bibitem{20} {\sc A. Defant, L. Frerick, J. Ortega-Cerd\`a , M. Ouna\"ies, and K. Seip}, The Bohnenblust-Hille inequality for homogeneous polynomials in hypercontractive, {\it Ann. of Math.} {\bf174} (2011), 512-517.
\bibitem{21} {\sc P. G. Dixon}, Banach algebras satisfying the non-unital von Neumann inequality, {\it Bull. London Math. Soc.} {\bf27} (4) (1995), 359-362.
\bibitem{22} {\sc P. B. Djakov and M. S. Ramanujan}, A remark on Bohr’s theorem and its generalizations, {\it J. Anal.} 8 (2000), 65-77.
\bibitem{23}  {\sc S. Evdoridis, S. Ponnusamy and A. Rasila}, Improved Bohr’s inequality for shifted disks, {\it Results Math.} {\bf76} (2021), 14.
\bibitem{24} {\sc S. R. Garcia, J. Mashreghi and W. T. Ross}, Finite Blaschke products and their connections, Springer, Cham, 2018.
\bibitem{25} {\sc I.R. Kayumov and S. Ponnusamy}, On a powered Bohr inequality, {\it Ann. Acad. Sci. Fenn. Ser. A}, {\bf44} (2019), 301-310.
\bibitem{26} {\sc M. S. Liu and S. Ponnusamy}, Multidimensional analogues of refined Bohr’s inequality, {\it Proc. Amer. Math. Soc.} {\bf149} (2021), 2133-2146.
\bibitem{RRS} {\sc R. Mandal, R. Biswas and S. K. Guin}, Geometric studies and the Bohr radius for certain normalized harmonic mappings, {\it Bull. Malays. Math. Sci. Soc.} {\bf47} (2024), 131. 
\bibitem{27} {\sc V. I. Paulsen, G. Popescu and D. Singh}, On Bohr’s inequality, {\it Proc. Lond. Math. Soc.} {\bf 85}(2) (2002), 493-512.
\bibitem{28} {\sc V. I. Paulsen and D. Singh}, Bohr’s inequality for uniform algebras, {\it Proc. Amer. Math. Soc.} {\bf132} (2004), 3577-3579.
\bibitem{29} {\sc G. Popescu}, Bohr inequalities for free holomorphic functions on polyballs, {\it Adv. Math.} {\bf347} (2019), 1002-1053.
\bibitem{30} {\sc S. Ponnusamy, R. Vijayakumar and K. J. Wirths}, New Inequalities for the Coefficients of Unimodular Bounded Functions, {\it Results Math.} {\bf75} (2020), 107. 
\end{thebibliography}
\end{document}